\documentclass[12pt]{amsart}
\usepackage{amssymb,latexsym,amsmath,amscd,graphicx}
\theoremstyle{plain}
\newtheorem{thm}{Theorem}[section]
\newtheorem{prop}[thm]{Proposition}
\newtheorem{cor}[thm]{Corollary}
\newtheorem{lemma}[thm]{Lemma}

\theoremstyle{definition}
\newtheorem{defn}[thm]{Definition}
\newtheorem{rmk}[thm]{Remark}
\newtheorem{rmks}[thm]{Remarks}

\newtheorem{ex}[thm]{Example}
\newtheorem{exs}[thm]{Examples}

\newcommand{\PP}{\mathbb{P}}

\newcommand{\NN}{\mathbb{N}}
\newcommand{\ZZ}{\mathbb{Z}}

\newcommand{\OO}{\mathcal{O}}

\newcommand{\I}{\mathcal{I}}

\newcommand{\OP}[1]{\OO_{\PP^{#1}}}

\newcommand{\Y}{\mathcal{Y}}
\newcommand{\X}{\mathcal{X}}
\newcommand{\Z}{\mathcal{Z}}
\newcommand{\mmax}{{\rm max}}

\newcommand{\hgt}{\mbox{ht}\;}

\newcommand{\Proj}{\mbox{Proj}\,}

\title{G-biliaison of ladder Pfaffian varieties}

\author{E. De Negri}
\address{Dipartimento di Matematica \\
Universit\`a di Genova, \hfil\break\indent via Dodecaneso 35,
16146 Genova, Italy}
\email{denegri@dima.unige.it}

\author{E. Gorla}
\address{Institut f\"ur Mathematik
\\ Universit\"at Z\"urich, \hfil\break\indent Winterthurerstrasse
190, CH-8057 Z\"urich, Switzerland}
\email{elisa.gorla@math.unizh.ch}



\thanks{The second author was partially supported by the Swiss
  National Science Foundation (grant no. 107887) and 
by the Max-Planck-Institut f\"ur Mathematik, Bonn. She also
wishes to express her gratitude to the Mathematics Department of the
University of Genova, where part of the work was done.}

\begin{document}

\maketitle

{\bf Abstract:} The ideals generated by pfaffians of mixed size
contained in a subladder of a skew-symmetric matrix of indeterminates
define arithmetically Cohen-Macaulay, projectively normal, reduced and
irreducible projective varieties. We show that these varieties belong
to the G-biliaison class of a complete intersection. In particular,
they are glicci.

\section*{Introduction}

Pfaffian ideals and the varieties that they define have been studied
both from the algebraic and from the geometric point of view.  In
\cite{a79} Avramov showed that the ideals generated by pfaffians of
fixed size define reduced and irreducible, projectively normal
schemes. In this article, we study the ideals generated by pfaffians
of mixed size contained in a subladder of a skew-symmetric matrix of
indeterminates. Ideals generated by pfaffians of the same size
contained in a subladder of a skew-symmetric matrix of indeterminates
were already studied by the first author. In~\cite{dn98} it is shown
that they define irreducible projective varieties, which are
arithmetically Cohen-Macaulay and projectively normal. A necessary and
sufficient condition for these schemes to be arithmetically Gorenstein
is given in terms of the vertices of the defining ladder. In
\cite{dn95}, one-cogenerated ideals of pfaffians are studied. The
deformation properties of schemes defined by pfaffians of fixed size
of a skew-symmetric matrix are studied in~\cite{kl78a} and~\cite{KL}.

In this paper, we study ladder ideals of pfaffians of mixed size from
the point of view of liaison theory (see \cite{mi98} for an
introduction to the subject, definitions and main results). A central
open question in liaison theory asks whether every arithmetically
Cohen-Macaulay projective scheme is glicci (i.e. whether it belongs to
the G-liaison class of a complete intersection of the same
codimension). Migliore and Nagel have shown that the question has an
affirmative answer up to deformation (see~\cite{mi02}). The main
result of this paper is that ladder pfaffian varieties belong to the
G-biliaison class of a linear variety. In particular they are glicci.
The result is a natural extension to ideals of pfaffians of the
results established by the second author in~\cite{go05u},
\cite{go05u2}, and~\cite{go06u} for ideals of minors.

In the first section we study the ideals generated by pfaffians of
mixed size contained in a subladder of a skew-symmetric matrix of
indeterminates. We prove that they define reduced and irreducible
projective schemes (see Proposition~\ref{primeCM}), that we call
ladder determinantal varieties. These varieties are shown to be
arithmetically Cohen-Macaulay and projectively normal in
Proposition~\ref{primeCM}. A localization argument is crucial for
extending these properties from the case of fixed size pfaffians of
ladders to the case when the pfaffians have mixed size (see
Proposition~\ref{local}). In Proposition~\ref{codim} we compute the
codimension of ladder pfaffian varieties.

Section 2 contains the liaison results. In Theorem~\ref{biliaison} we
prove that ladder pfaffian varieties belong to the G-biliaison class
of a linear variety. Using standard liaison results, we conclude in
Corollary~\ref{glicci} that they are (evenly) G-linked to a complete
intersection.

\section{Pfaffian ideals of ladders and ladder pfaffian varieties}

Let $V$ be a variety in $\PP^r=\PP^r_K$, where $K$ is an algebraically
closed field of arbitrary characteristic. Let $I_V$ be the
saturated homogeneous ideal
associated to $V$ in the coordinate ring
of $\PP^r$. 
Let $\I_V\subset \OP{r}$ be the ideal sheaf of~$V$.
Let $W$ be a scheme that contains $V$. We denote by $\I_{V|W}$
the ideal sheaf of $V$ restricted to $W$, i.e. the quotient sheaf $\I_V/\I_W$.

Let $X=(x_{ij})$ be an $n\times n$ skew-symmetric matrix of indeterminates.
In other words, the entries
$x_{ij}$ with $i<j$ are indeterminates, $x_{ij}=-x_{ji}$ for
$i>j$, and $x_{ii}=0$ for all $i=1,...,n$.
Let $K[X]=K[x_{ij} \;|\; 1\leq i<j\leq n ]$ be the polynomial ring associated 
to $X$. Given a nonempty subset ${\mathcal U} =\{u_1,...,u_{2p} \}$  of  
$\{1,...,n\}$ we denote by $[u_1,\cdots,u_{2p}]$ 
the {\em pfaffian} of the matrix $(x_{ij})_{i\in{\mathcal U},j\in {\mathcal U}}$. 

\begin{defn}\label{ladd}
A {\em ladder} $\mathcal Y$ of $X$ is a subset of the set
 $\{(i,j)\in\NN^2 \;|\; 1\le i,j\le n\}$
with the following properties :
\begin{enumerate}
\item if $(i,j)\in {\mathcal Y}$ then $(j,i)\in {\mathcal Y}$,
\item if  $i<h,j>k$ and $(i,j),(h,k)$ belong to $\mathcal Y$, then also
$(i,k),(i,h),(h,j),(j,k)$ belong to $\mathcal Y$.
\end{enumerate}

\begin{figure}
\centering%
\includegraphics[height=20cm , width=10cm,keepaspectratio]  {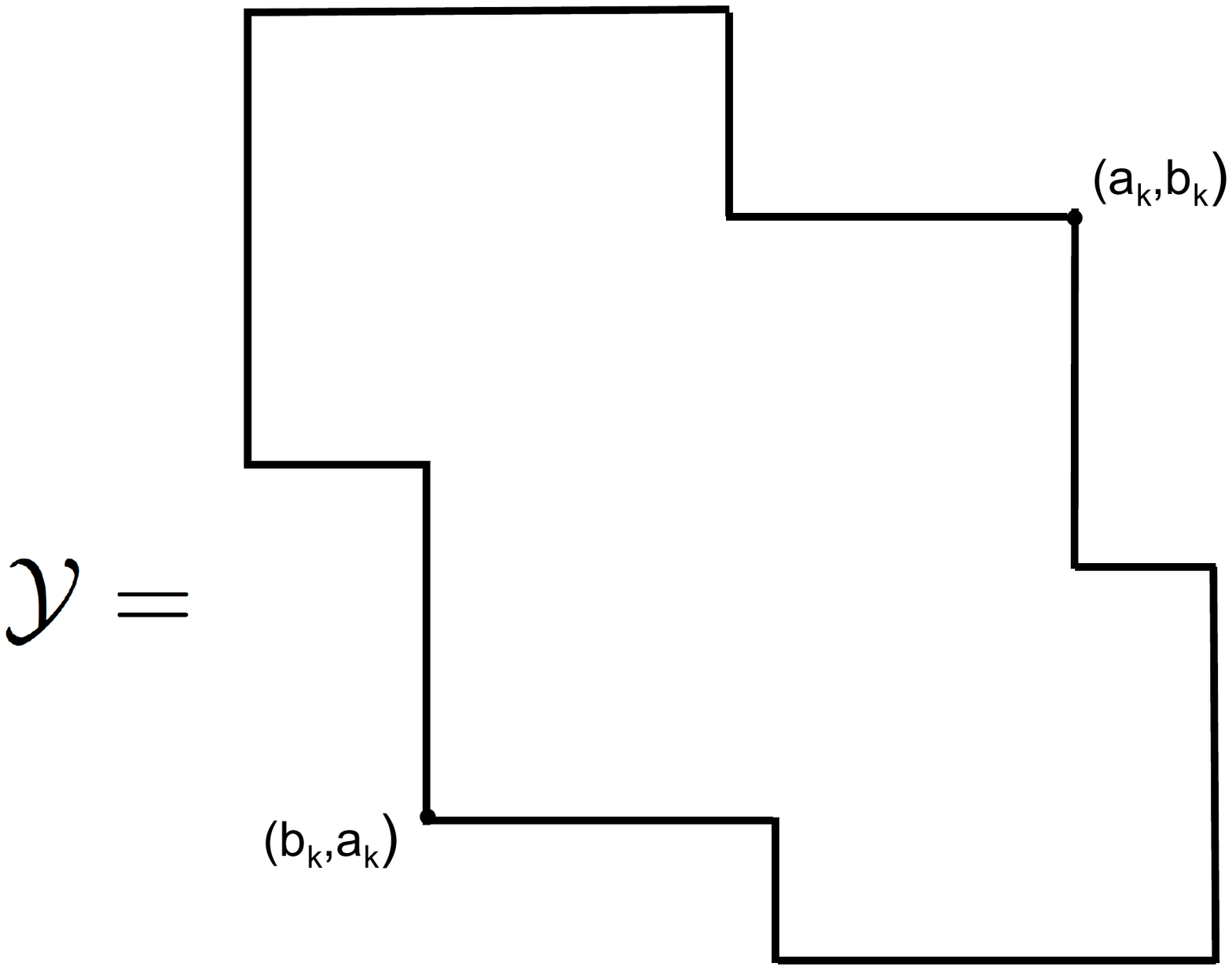}
\caption{A Ladder}
\end{figure}

We do not assume that a ladder $\Y$ is connected, nor that $X$ is the 
smallest skew-symmetric matrix
having $\Y$ as ladder. 

It is easy to see that any ladder can be decomposed as a union of square subladders
\begin{equation}\label{decomp}
\Y=\X_1\cup\ldots\cup \X_s
\end{equation} 
where
 $$\X_k=\{(i,j)\;|\; a_k\le i,j \le b_k\},$$ 
for some  integers $1\leq a_1\leq\ldots\leq a_s\leq n$ and $1\leq b_1\leq\ldots\leq
b_s\leq n$ such that $a_k<b_k$ for all $k$.
We say that $\Y$ is the ladder with {\em upper
corners} $(a_1,b_1),\ldots,(a_s,b_s)$, and that $\X_k$ is the
square subladder of $\Y$ with upper outside corner $(a_k,b_k)$.
We allow two upper corners to have the
same first or second coordinate, however we assume that no two
upper corners coincide. Notice that with this convention a
ladder does not have a unique decomposition of the form
(\ref{decomp}). In other words, a ladder does not
correspond uniquely to a set of
upper corners $(a_1,b_1),\ldots,(a_s,b_s)$. However, a ladder is
determined by its upper corners as in (\ref{decomp}). 
Moreover, the upper corners of a ladder $\mathcal Y$ determine both the subladders 
$\X_k$ and the smallest skew-symmetric submatrix of $X$ that 
has $\Y$ as ladder. 

Given a ladder $\mathcal Y$ we set $Y=\{x_{ij}\in X\;|\; (i,j)\in
{\mathcal Y},\; i<j\}$, and denote by  $K[Y]$ the polynomial
ring $K[x_{ij}\;|\; x_{ij}\in Y]$. If $p$ is a positive integer, we let  
$I_{2p}(Y)$ denote the ideal generated by the set of the
$2p$-pfaffians of $X$ which involve only indeterminates of
$Y$. In particular $I_{2p}(X)$ is the ideal of $K[X]$ generated by 
the $2p$-pfaffians of $X$.
\end{defn}

Whenever we consider a ladder $\Y$, we assume that it comes with its
set of upper corners and the corresponding decomposition as a
union of square subladders as in (\ref{decomp}). 

Notice that the set of upper corners as given in our definition
contains all the usual upper outside corners, and may contain some of
the usual upper inside corners, as well as other elements of the
ladder which are not corners of the ladder in the usual sense. 

\begin{ex}\label{excorners}
Consider the set of upper corners $\{(1,2), (1,4), (3,4), (3,6),(4,7)\}$. Then
  $k=5$ and $$\X_1=\begin{array}{cc}
(1,1) & (1,2) \\
(2,1) & (2,2)\end{array}\;\;\;
\X_2=\begin{array}{cccc}
(1,1) & (1,2) & (1,3) & (1,4) \\
(2,1) & (2,2) & (2,3) & (2,4) \\
(3,1) & (3,2) & (3,3) & (3,4) \\
(4,1) & (4,2) & (4,3) & (4,4)
\end{array}\;\;\;
\X_3=\begin{array}{cc}
(3,3) & (3,4) \\
(4,3) & (4,4) \end{array}$$ 
$$\X_4=\begin{array}{cccc}
(3,3) & (3,4) & (3,5) & (3,6) \\
(4,3) & (4,4) & (4,5) & (4,6) \\
(5,3) & (5,4) & (5,5) & (5,6) \\
(6,3) & (6,4) & (6,5) & (6,6)
\end{array}\;\;\; \X_5=\begin{array}{cccc}
(4,4) & (4,5) & (4,6) & (4,7) \\
(5,4) & (5,5) & (5,6) & (5,7) \\
(6,4) & (6,5) & (6,6) & (6,7) \\
(7,4) & (7,5) & (7,6) & (7,7)
\end{array}$$
The ladder determined by this choice of upper corners is 
$$\Y=\X_1\cup\X_2\cup\X_3\cup\X_4\cup\X_5=\begin{array}{ccccccc}
(1,1) & (1,2) & (1,3) & (1,4) & & & \\
(2,1) & (2,2) & (2,3) & (2,4) & & & \\
(3,1) & (3,2) & (3,3) & (3,4) & (3,5) & (3,6) & \\
(4,1) & (4,2) & (4,3) & (4,4) & (4,5) & (4,6) & (4,7) \\
 & & (5,3) & (5,4) & (5,5) & (5,6) & (5,7) \\
 & & (6,3) & (6,4) & (6,5) & (6,6) & (6,7) \\
 & & & (7,4) & (7,5) & (7,6) & (7,7)
\end{array}$$
$(1,2)$ is the upper outside corner of $\X_1$, $(1,4)$ is the upper
outside corner of $\X_2$, $(3,4)$ is the upper outside corner of
$\X_3$, $(3,6)$ is the upper outside corner of $\X_4$, and $(4,7)$ is
the upper outside corner of $\X_5$.

Notice that our set of upper corners contains $(3,4)$, which in the
usual terminology is referred to as an upper inside corner. However it
does not contain the usual upper inside corner $(4,6)$. Moreover, our
set of upper corners
contains $(1,2)$ which is not a corner in the usual terminology. It
contains also all the usual upper outside
corners, namely $(1,4),(3,6),$ and $(4,7)$.
Let $$X=\left(\begin{array}{ccccccc}
0 & x_{1,2} & x_{1,3} & x_{1,4} & x_{1,5} & x_{1,6} & x_{1,7} \\
-x_{1,2} & 0 & x_{2,3} & x_{2,4} & x_{2,5} & x_{2,6} & x_{2,7} \\
-x_{1,3} & -x_{2,3} & 0 & x_{3,4} & x_{3,5} & x_{3,6} & x_{3,7} \\
-x_{1,4} & -x_{2,4} & -x_{3,4} & 0 & x_{4,5} & x_{4,6} & x_{4,7} \\
-x_{1,5} & -x_{2,5} & -x_{3,5} & -x_{4,5} & 0 & x_{5,6} & x_{5,7} \\
-x_{1,6} & -x_{2,6} & -x_{3,6} & -x_{4,6} & -x_{5,6} & 0 & x_{6,7} \\
-x_{1,7} & -x_{2,7} & -x_{3,7} & -x_{4,7} & -x_{5,7} & -x_{6,7} & 0
\end{array}
\right)$$ be the smallest skew-symmetric matrix having $\Y$ as ladder.
The set of indeterminates corresponding to $\Y$ is
$$Y=\begin{array}{ccccccc}
& x_{1,2} & x_{1,3} & x_{1,4} & & & \\
& & x_{2,3} & x_{2,4} & & & \\
& & & x_{3,4} & x_{3,5} & x_{3,6} & \\
& & & & x_{4,5} & x_{4,6} & x_{4,7} \\
& & & & & x_{5,6} & x_{5,7} \\
& & & & & & x_{6,7}
\end{array}$$
\end{ex}

\begin{defn}\label{ideal}
Let $\Y=\X_1\cup\ldots\cup \X_s$ be a ladder as in
Definition~\ref{ladd}.\newline
Let $X_k=\{x_{i,j}\;|\; (i,j)\in\X_k,\; i<j\}$ for $k=1,\dots,s$.
Fix a vector $t=(t_1,\ldots,t_s)$, $t\in \{1,\ldots,\lfloor\frac{n}{2}\rfloor\}^s$.
The {\em pfaffian ideal} $I_{2t}(Y)$ is by definition the sum of 
pfaffian ideals $I_{2t_1}(X_1)+\ldots+I_{2t_s}(X_s)\subseteq
K[Y]$. Sometimes we refer to these ideals as {\em pfaffian
ideals of ladders}.  
\end{defn}

\begin{ex}
Let $\Y$ be the ladder of Example~\ref{excorners}, together with the same
choice of upper corners. Let $t=(1,2,1,2,2)$, then the pfaffian ideal is
$$I_{2t}(Y)=(x_{1,2},-x_{1,3}x_{2,4}+x_{1,4}x_{2,3},x_{3,4},
-x_{3,5}x_{4,6}+x_{3,6}x_{4,5},x_{4,5}x_{6,7}-x_{4,6}x_{5,7}+x_{4,7}x_{5,6})$$
$$\subseteq
K[x_{1,2},x_{1,3},x_{1,4},x_{2,3},x_{2,4},x_{3,4},x_{3,5},x_{3,6},x_{4,5},
x_{4,6},x_{4,7},x_{5,6},x_{5,7},x_{6,7}].$$
$I_{2t}(Y)$ is the saturated ideal of a variety of codimension 5 in $\PP^{13}$.
\end{ex}

\begin{rmks}\label{conv}
\begin{enumerate}
\item Let  $\mathcal Z\supseteq \mathcal Y$ be two ladders of $X$, 
and let $Z,Y$ be 
the corresponding sets of indeterminates. We have an isomorphism of
  graded $K$-algebras 
$$K[Y]/I_{2t}(Y)\cong K[Z]/I_{2t}(Y)+(x_{ij}\;|\; x_{ij}\in
Z\setminus Y)\cong K[X]/I_{2t}(Y)+(x_{ij}\;|\; x_{ij}\in
X\setminus Y).$$
Here $I_{2t}(Y)$ is an ideal in $K[Y],K[X],$ and $K[Z]$ respectively. 
Then the height of the ideal $I_{2t}(Y)$ does not depend of
whether we regard it as an ideal of $K[Y],K[X],$ or $K[Z]$.
\item We can assume without loss of generality that $$2t_k\leq
  b_k-a_k+1.$$ In fact, if $2t_k>b_k-a_k+1$ then
  $I_{2t_k}(X_k)=0$.
  \item Moreover, we can assume that
  $$a_k-a_{k-1}>t_{k-1}-t_k \;\;\;\mbox{and}\;\;\;
  b_k-b_{k-1}>t_k-t_{k-1}, \;\;\;\;\mbox{for
    $k=2,\ldots,s$.}$$
If $a_k-a_{k-1}\leq t_{k-1}-t_k$, by successively developing each $2t_{k-1}$-pfaffian 
of $X_{k-1}$ with respect to the first $2(a_k-a_{k-1})$ rows and columns, 
we obtain an expression of the pfaffian as a combination of pfaffians of 
size $2t_{k-1}-2(a_k-a_{k-1})\geq 2t_k$ that involve only rows and columns 
from $X_k$. Therefore $I_{2t_k}(X_k)\supseteq I_{2t_{k-1}}(X_{k-1})$.
Similarly, if $b_k-b_{k-1}\leq t_k-t_{k-1}$, by developing each $2t_k$-pfaffian 
of $X_k$ with respect to the last $2(b_k-b_{k-1})$ rows and columns, 
we obtain an expression of the pfaffian as a combination of pfaffians of 
size $2t_k-2(b_k-b_{k-1})\geq 2t_{k-1}$ that involve only rows and columns 
from $X_{k-1}$. Therefore $I_{2t_k}(X_k)\subseteq I_{2t_{k-1}}(X_{k-1})$.
In either case, we can remove a part of the ladder and reduce to the
study of a proper subladder that corresponds to the same
pfaffian ideal. 
For example, if $b_k-b_{k-1}\leq t_k-t_{k-1}$ we can
consider the ladder 
$$\tilde\Y=\X_1\cup\ldots\cup \X_{k-1}\cup
\X_{k+1}\cup\ldots\cup \X_s$$ and let 
$$t'=(t_1,\ldots,t_{k-1},t_{k+1},\ldots,t_s).$$ 
Since  $I_{2t_k}(X_k)\subseteq
I_{2t_{k-1}}(X_{k-1})$, we have $I_{2t}(Y)=I_{2t'}(\tilde Y),$ where
$\tilde Y$ is the set of indeterminates corresponding to $\tilde\Y$.
\end{enumerate}
\end{rmks}

The class of pfaffian ideals that we consider is very large. We now give
examples of interesting families of ideals generated by pfaffians,
which belong to the class of pfaffian ideals that we study.

\begin{exs}
\begin{enumerate}
\item
If $t=(t,\ldots,t)\in\{1,\ldots,\lfloor\frac{n}{2}\rfloor\}^s$
then $I_{2t}(Y)$ is the
ideal generated by the pfaffians of size $2t$ of $X$ that
involve only indeterminates from $Y$. In  \cite{dn98} it is proven
that in this case $K[Y]/I_{2t}(Y)$ is  a Cohen-Macaulay normal domain.

\item
If we choose all the upper corners on the same row, we obtain ideals
of pfaffians of matrices which are contained one in the other. In this
case $1=a_1=a_2=\ldots=a_s$, hence $1<b_1<b_2<\ldots<b_s=n$. By
Remark~\ref{conv} (3), this forces $1\leq t_1<t_2<\ldots<t_s$. The
ideal $I_{2t}(X)$ is generated by the $2t_i$-pfaffians of the
submatrix of the first $b_i$ rows and columns, $i=1,...,s$.

Similarly, if we choose all the upper corners on the same column, that
is, $1=a_1<a_2<\ldots<a_s$ and $b_1=b_2=\ldots=b_s=n$,  we obtain the
ideal generated by the $2t_i$-pfaffians of the last $n-a_i+1$ rows and
columns, $i=1,...,s$, with $t_1>t_2>\ldots>t_s\geq 1$.
Notice that these two choices of upper corners produce the same
family of ideals.

\item  Consider the ladder with two upper corners $(1,b),(1,n)$,
  $b<n$, and the vector $(t_1,t_2)=(t,t+1)$. Then the  ideal
  $I_{2t}(Y)$ is generated by all the $2t+2$-pfaffians  and by the
  $2t$-pfaffians of the first $b$ rows and columns, of an $n\times n$
  skew-symmetric matrix of indeterminates. 
This ideals belong to a well-known class of ideals generated by
pfaffians in a matrix, the cogenerated ideals, which have been studied
in  \cite{dn95}. In fact $I_{2t}(Y)$ is the ideal cogenerated by the
pfaffian $[1,2,\dots, 2t-1,b+1]$. 
Notice however that not every one-cogenerated pfaffian ideal is a
pfaffian ideal of ladders.
\end{enumerate}
\end{exs}

We will show that for every vector $t$ the ideals $I_{2t}(Y)$ are prime (see
Proposition~\ref{primeCM}). Therefore they define reduced and
irreducible projective varieties.

\begin{defn}\label{pfaff}
Let $V\subseteq\PP^r$. $V$ is a
{\em pfaffian variety} if $I_V=I_{2t}(X)$, where $X$ is a skew
symmetric matrix of indeterminates of size $n\times n$.
$V$ is a {\em ladder pfaffian variety} if
$I_V=I_{2t}(Y)=I_{2t_1}(X_1)+\ldots+I_{2t_s}(X_s)$ for some ladder  
$\Y=\X_1\cup\ldots\cup \X_s$ and some  vector 
$t=(t_1,\ldots,t_s)\in \{1,\ldots,\lfloor\frac{n}{2}\rfloor\}^s$.
\end{defn}

Notice that every pfaffian variety is a ladder pfaffian
variety. Therefore, from now on we will only consider ladder
pfaffian varieties. Moreover, in view of Remark~\ref{conv} (1) 
we will not distinguish between
ladder pfaffian varieties and cones over them.

In this section we study ladder pfaffian varieties. We prove that
their saturated ideals are generated by pfaffians of mixed size
contained in a subladder of a skew-symmetric matrix of indeterminates,
or in other words that the ideals in question are prime.  We prove
that the ladder pfaffian varieties are arithmetically Cohen-Macaulay
and projectively normal, and we compute their codimension. We choose
to follow a classical commutative algebra localization argument to
approach the problem. Some of our results could be obtained also by
using a Schubert calculus approach, at least for the case of ideals of
pfaffians of fixed size in a matrix, which define Schubert varieties
in orthogonal Grassmannians.

In order to establish the properties that we just mentioned, we will
make use of a localization argument (analogous to that of Lemma~7.3.3
in~\cite{br93}). The following proposition will be
crucial in the sequel. We use the notation of Definition~\ref{ladd}
and Definition~\ref{ideal}, and refer to Fig. 2.  Notice that if
$t_l\geq 2$ for some $l$, then it is always possible to choose a $k$
such $t_k\geq 2$, $a_{k+1}-1\geq a_k$, and $b_k\geq b_{k-1}+1$. In
fact, it suffices to choose $k$ such that $t_k\geq t_l$ for all $l$,
and the inequalities follow from Remark~\ref{conv} (3). Notice
moreover that for a classical ladder (i.e. a ladder for which no two
vertices belong to the same row or column) these conditions are
automatically satisfied.

\begin{prop}\label{local}
Let $\mathcal Y=\X_1\cup\dots\cup \X_s$ be a ladder of a 
skew-symmetric matrix $X$ of indeterminates. Let
$t=(t_1,\ldots,t_s)\in\{1,\ldots,\lfloor\frac{n}{2}\rfloor\}^s$, 
and let $I_{2t}(Y)$ be the corresponding pfaffian ideal.
Fix $k\in\{1,\ldots,s\}$ such that $t_k\geq 2$, $a_{k+1}-1\geq a_k$,
and $b_k\geq b_{k-1}+1$. Let 
$t'=(t_1,\ldots,t_{k-1},t_k-1,t_{k+1},\ldots,t_s).$ 
Let  $\mathcal Y'$ be the subladder of  $\mathcal Y$ 
with outside corners
$$(a_1,b_1),\dots,(a_{k-1},b_{k-1}),(a_k+1,b_k-1),(a_{k+1},b_{k+1}),\dots,(a_s,b_s).$$


Then there is an isomorphism $$K[Y]/I_{2t}(Y)[x_{a_k,b_k}^{-1}]\cong 
K[Y']/I_{2t'}(Y')[x_{a_k,b_{k-1}+1},\ldots,x_{a_k,b_k},x_{a_k+1,b_k},\ldots,
x_{a_{k+1}-1,b_k}][x_{a_k,b_k}^{-1}].$$
\end{prop}

\begin{figure}
\centering
\includegraphics[height=20cm,width=10cm,keepaspectratio]{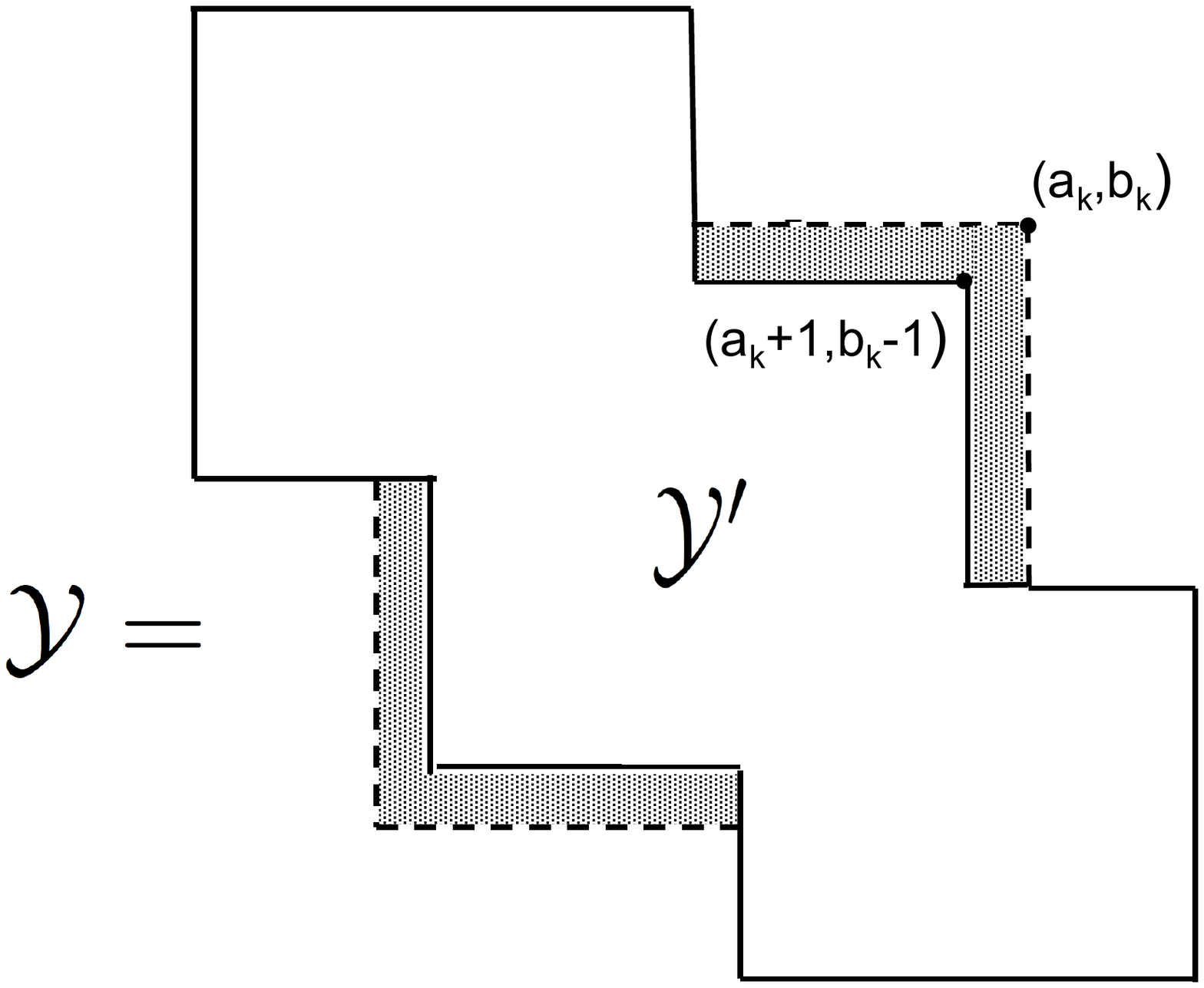}
\caption{}
\end{figure}

\begin{proof}
We prove the proposition in the case $s=k=2$. In the general case the
proof works exactly the same way. 
Let $\Y=\X_1\cup \X_2$ be the ladder with upper corners
$(1,c)$ and $(a,b)$,  $t=(t_1,t_2)$ with $t_2\ge 2$. 
Let $X_1=(x_{i,j})_{1\leq i,j\leq c}$ and $X_2=(x_{i,j})_{a\leq
  i,j\leq b}$ be two submatrices of $X$.
Let  $\tilde{X}_1=(\tilde x_{u,v})$ be a skew-symmetric 
matrix of indeterminates of size $c\times c$, whose entries have indexes
$1\leq u,v\leq c$,  where $\tilde x_{u,v}=x_{u,v}$ if $u<a$ or $v<a$, and  $\tilde x_{u,v}$ are new indeterminates 
whenever $a\le u<v\le c$ (that is  $(u,v)\in \X_2$). 

Let $\tilde{Y}$ be the set of all the indeterminates in  the matrices
$\tilde{X}_1$ and $X_2$, and denote by  
$L$ the ideal $(\{x_{u,v}-\tilde x_{u,v}\}_{a \leq u<v\leq c})$. Then 
$$K[Y]/I_{2t_1}(X_1)+I_{2t_2}(X_2)\cong 
K[\tilde{Y}]/I_{2t_1}(\tilde{X}_1)+I_{2t_2}(X_2)+L.$$
Therefore 
$$K[Y]/(I_{2t_1}(X_1)+I_{2t_2}(X_2))[x_{a,b}^{-1}]\cong
K[\tilde{Y}]/(I_{2t_1}(\tilde{X}_1)+I_{2t_2}(X_2)+
L) [x_{a,b}^{-1}]\cong$$
$$K[\tilde{Y}][x_{a,b}^{-1}]/I_{2t_1}(\tilde{X}_1)^e+I_{2t_2}(X_2)^e+L^e$$
where for any ideal $I$, $I^e$ denotes the extension $IK[\tilde{Y}][x_{a,b}^{-1}]$.
By a well-know localization argument (see for example \cite[Proposition~3.2]{dn98}) we have that 
$$K[X_2][x_{a,b}^{-1}]/I_{2t_2}(X_2)^e\cong 
K[X_2][x_{a,b}^{-1}]/I_{2t_2-2}(X'_2)^e,$$
where $X'_2$ is the submatrix of  $X_2$ obtained by removing the 
$a$-th and the $b$-th row and column. 
Therefore 
$$K[\tilde{Y}][x_{a,b}^{-1}]/
I_{2t_1}(\tilde{X}_1)^e+I_{2t_2}(X_2)^e+
L^e\cong K[\tilde{Y}][x_{a,b}^{-1}]/
I_{2t_1}(\tilde{X}_1)^e+I_{2t_2-2}(X'_2)^e+
L^e\cong $$
$$K[Y]/I_{2t'}(Y')[x_{a,b}^{-1}]\cong K[Y']/I_{2t'}(Y')
[x_{a,c+1},\ldots,x_{a,b},x_{a+1,b}\ldots,
x_{b-1,b}][x_{a,b}^{-1}].$$
\end{proof}

Using Proposition~\ref{local} we can establish some
properties of ladder pfaffian varieties. 

\begin{prop}\label{primeCM}
Pfaffian ideals of ladders define reduced and irreducible,
arithmetically Cohen-Macaulay projectively normal varieties.
\end{prop}

\begin{proof}
Let $I_{2t}(Y)$ be a pfaffian ideal.
Let $t_{\mbox{\tiny
    max}}$ be the maximum of $\{t_1,...,t_s\}$. If $t_{\mbox{\tiny
    max}}=1$ then $I_{2t}(X)$ is generated by indeterminates, und we are done. Assume that $t_{\mbox{\tiny
    max}}\ge 2$. 
Let ${\mathcal {\widehat Y}}$ be the ladder obtained by enlarging
${\mathcal Y}$ along its borders by the  region which increases the
size of every $\X_k$ by $t_{\mbox{\tiny max}}-t_k$. 
Thus ${\mathcal {\widehat Y}}$ is the ladder with upper
corners $(a_k-t_{\mbox{\tiny max}}+t_k,b_k+t_{\mbox{\tiny max}}-t_k)$,
with  $k=1,\ldots,s$.  Let ${\widehat Y}=\{x_{ij}\in X\;|\; (i,j)\in
{\mathcal {\widehat Y}},\; i<j\}$ and let $\Psi={\widehat  Y}\setminus
Y$.
By   Proposition \ref{local}, we can repeatedly localize $K[{\widehat
 Y}]/I_{2t_{\mbox{\tiny max}}}({\widehat  Y})$ at appropriate upper
outside corners and obtain  the original ladder $\Y$ and the pfaffians
of size $t_1,\ldots,t_s$. It follows that there exists 
a subset $\{z_1,...,z_p\}$ of $\Psi$
such that 
$$K[{\widehat  Y}]/I_{2t_{\mbox{\tiny max}}}({\widehat  Y})[z_1^{-1},\ldots, z_p^{-1}]\cong 
K[Y]/I_{2t}(Y)[\Psi][z_1^{-1},\ldots, z_p^{-1}].$$
By \cite[1.2,2.1,3.5]{dn98} one has that $K[{\widehat
  Y}]/I_{2t_{\mbox{\tiny max}}}({\widehat  Y})$ is a Cohen-Macaulay
normal domain, thus  $K[Y]/I_{2t}(Y)[\Psi][u_1^{-1},\ldots, u_q^{-1}]
$ is a Cohen-Macaulay normal domain. Since  $\Psi$ is a set of
indeterminates over $K[Y]/I_{2t}(Y)$ and
$\{u_1,....,u_q\}\subset\Psi$, then also $K[Y]/I_{2t}(Y)[\Psi]$ is a
Cohen-Macaulay normal domain. Hence  $I_{2t}(Y)$ defines a
reduced and irreducible, arithmetically Cohen-Macaulay normal projective
variety.
\end{proof}

A standard argument allows us to compute the codimension of
ladder pfaffian varieties. The notation is the same as in
Proposition~\ref{local}.

\begin{prop}\label{codim}
 Let $\Y$ be a ladder with  upper corners $(a_1,b_1),\ldots,(a_s,b_s)$.
Let  
$${\mathcal L}=\{(i,j)\;|\; a_k+t_k-1\leq i,j\leq b_k-t_k+1,
\;\mbox{for some}\; 1\leq k\leq s\}$$
be a subset of $\Y$. Then ${\mathcal L}$ is a ladder
and the height of $I_{2t}(Y)$ is equal to the cardinality of 
$\{(i,j)\in{\mathcal L} \;|\; i<j\}$.
\end{prop}

\begin{proof}
Observe that $$a_k+t_k-1>a_{k-1}+t_{k-1}-1,\;\;\;\mbox{and}\;\;\; 
b_k-t_k+1>b_{k-1}-t_{k-1}+1$$ 
by Remark~\ref{conv} (3). Moreover by Remark~\ref{conv}
(2) we have
$b_k-a_k>2t_k-2$. Then $$b_k-t_k+1>a_k+t_k-1$$ for all $k$.
Therefore ${\mathcal L}$ is a ladder with upper corners 
$\{(a_k+t_k-1,b_k-t_k+1)\;|\;k=1,\ldots,s\}.$ Notice that $\mathcal
L$ has no two corners on the same row or column. Let
$L=\{x_{i,j}\;|\; (i,j)\in {\mathcal L},\; i<j\}$.

We argue by induction on 
$\tau=t_1+\ldots+t_s\ge s$. If  $\tau=s$, then $t_1=\ldots=t_s=1$, and
${\mathcal L}=\Y$. Moreover,  
$$I_{2}(Y)=(x_{ij}\;|\;x_{ij}\in Y,\; i<j)=(x_{ij}\;|\;x_{ij}\in L,\;
i<j),$$ thus the thesis holds true. 

Assume then that the thesis is true for $\tau-1\ge s$ and prove it for
$\tau $. Since $\tau>s$, then  $t_k\geq 2$ for some $k$. By
Proposition~\ref{local} we have an
isomorphism $$K[Y]/I_{2t}(Y)[x_{a_k,b_k}^{-1}]\cong 
K[Y']/I_{2t'}(Y')[x_{a_k,b_{k-1}+1},\ldots,x_{a_k,b_k},x_{a_k,b_k+1},\ldots,
x_{a_k,b_{k+1}-1}][x_{a_k,b_k}^{-1}].$$
Since $x_{a_k,b_k}$ does not divide zero modulo
$I_{2t'}(Y')$ and $I_{2t}(Y)$, we have
$$\hgt I_{2t}(Y)=\hgt I_{2t'}(Y'). $$  
Notice that the same ladder $\mathcal L$ computes the height of both 
$I_{2t'}(Y')$ and $I_{2t}(Y)$, thus the thesis follows by the
induction hypothesis. 
\end{proof}

\section{Linkage of ladder pfaffian varieties}

In this section we prove that ladder pfaffian
varieties belong to the G-biliaison class of a complete
intersection. The biliaisons are performed on ladder pfaffian
varieties, which are reduced and irreducible (hence
generically Gorenstein), and arithmetically Cohen-Macaulay.
Therefore we can conclude that ladder pfaffian varieties are
glicci. Notice the analogy with determinantal varieties, symmetric 
determinantal varieties
and mixed ladder determinantal varieties, that were treated by
the second author with analogous techniques in~\cite{go05u},
\cite{go05u2}, and~\cite{go06u}.

The following  lemma due  to 
De Concini and Procesi \cite [6.1]{dp76}  will be needed in the
sequel.

\begin{lemma}\label{relazioni}
Let $A$ be a skew symmetric $n\times n$ matrix, $p,m\leq n$ even
integers and $c_1,...,c_p,d_1,...,d_m$ elements of the
set $\{1,...,n\}$.
Then
$$[c_1,...,c_p][d_1,...,d_m]-\sum_{h=1}^{p}
[c_1,...,c_{h-1},d_1,c_{h+1},...,c_p]
[c_h,d_2,...,d_m]=$$
$$\sum_{k=2}^{m}(-1)^{k-1}[d_k,d_1,c_1,...,c_p][d_2,...,d_{k-1},d_{k+1},...,d_m]$$
where $[...]$ denotes a pfaffian of $A$.
\end{lemma}

The following result will also be needed in the proof. We
will use it to construct the ladder pfaffian varieties on
which we perform the G-biliaisons. We follow the notation established
in Definitions~\ref{ladd} and~\ref{ideal}.

\begin{lemma}\label{codimW}
Let $V\subseteq\PP^r$ be a ladder pfaffian variety of codimension $c$.
Let $\mathcal Y$ be the ladder corresponding to $V$, and assume that
$t_k=\mmax\{t_1,\ldots,t_s\}\geq 2$. Let $\Z$ be the subladder of $\Y$
with upper corners $$(a_1,b_1),\ldots,(a_{k-1},b_{k-1}),(a_k,b_k-1),
(a_k+1,b_k),(a_{k+1},b_{k+1}),\ldots,(a_s,b_s)$$ and let 
$u=(t_1,\ldots,t_{k-1},t_k,t_k,t_{k+1},\ldots,t_s)$. Then the ladder
pfaffian variety $W\subseteq\PP^r$ with $I_W=I_{2u}(\Z)$ 
has codimension $c-1$.
\end{lemma}

\begin{proof}
We decompose the ladder $\Z$ as
$$\Z=\X_1\cup\ldots\cup \X_{k-1}\cup \X_k^{(1)}\cup \X_k^{(2)} 
\cup \X_{k+1}\cup\ldots\cup \X_s$$ where $\X_k^{(1)}$, $\X_k^{(2)}$
are the square subladders  with upper outside corner
$(a_k,b_k-1)$ and $(a_k+1,b_k)$, respectively. Let 
$u=(t_1,\ldots,t_{k-1},t_k,t_k,t_{k+1},\ldots,t_s),\;\;\;$ $u\in
\{1,\ldots,\lfloor\frac{n}{2}\rfloor\}^{s+1}.$

If the ladder $\Z$ satisfies the inequalities of
Remark~\ref{conv} (2) and (3), then the codimension count follows from
Proposition~\ref{codim}. In fact, the codimension $c$ of $V$
equals the cardinality  of the subset 
$\{(i,j)\in{\mathcal L}\;|\; i<j\}$ where ${\mathcal L}$ is 
the ladder with  upper
corners $(a_1+t_1-1,b_1-t_1+1),\ldots,(a_s+t_s-1,b_s-t_s+1).$ 
The codimension of $W$ equals cardinality of 
$\{(i,j)\in {\mathcal L'}\;|\; i<j\}$, where
${\mathcal L}'$ is the ladder obtained from ${\mathcal L}$ by removing 
$(a_k+t_k-1,b_k-t_k+1)$ and 
$(b_k-t_k+1,a_k+t_k-1)$. So we conclude that $W$ has codimension $c-1$.

Notice however that the ladder $\Z$ may not satisfy the inequalities of
Remark~\ref{conv} (2),(3) even under the assumption that the ladder $\Y$ does. 
In particular, the following three situations may occur:
\begin{enumerate}
\item $2t_k=b_k-a_k+1>b_k-a_k=(b_k-1)-a_k+1=b_k-(a_k+1)+1$, 
\item $a_{k+1}-(a_k+1)=t_k-t_{k+1}$,
\item $(b_k-1)-b_{k-1}=t_k-t_{k-1}$.
\end{enumerate}

In case (1) we delete the subladders $\X_k^{(1)}$ and $\X_k^{(2)}$,
and let $$\Z=\X_1\cup\ldots\cup \X_{k-1}\cup\X_{k+1}\cup\ldots\cup\X_s$$
and $u=(t_1,\ldots,t_{k-1},t_{k+1},\ldots,t_s).$

In case (2) we delete the subladder $\X_k^{(2)}$, and let 
$$\Z=\X_1\cup\ldots\cup \X_{k-1}\cup\X_k^{(1)}\cup\X_{k+1}\cup\ldots\cup\X_s$$
and $u=(t_1,\ldots,t_{k-1},t_k,t_{k+1},\ldots,t_s).$

In case (3) we delete the subladder $\X_k^{(1)}$, and let 
$$\Z=\X_1\cup\ldots\cup \X_{k-1}\cup\X_k^{(2)}\cup\X_{k+1}\cup\ldots\cup\X_s$$
and $u=(t_1,\ldots,t_{k-1},t_k,t_{k+1},\ldots,t_s).$

Notice that it may happen that more than one of the cases (1), (2) and
(3) is verified for the ladder $\Z$. In this case, we behave as if we
were in the situation (1). As we already observed, none of the
operations above affects the ideal $I_W$.

If we are in situation (1), then $2t_k=b_k-a_k+1$. Applying
Proposition~\ref{codim} to the ladder 
$\Z=\X_1\cup\ldots\cup \X_{k-1}\cup\X_{k+1}\cup\ldots\cup\X_s$ we
obtain that the codimension of $W$ equals the cardinality of 
$\{(i,j)\in {\mathcal L'}\;|\; i<j\}$, where
${\mathcal L}'$ is the ladder with upper corners
$$(a_1+t_1-1,b_1-t_1+1),\ldots,(a_{k-1}+t_{k-1}-1,b_{k-1}-t_{k-1}+1),$$
$$(a_{k+1}+t_{k+1}-1,b_{k+1}-t_{k+1}+1),\ldots,
(a_s+t_s-1,b_s-t_s+1).$$
Since $a_k+t_k-1=b_k-t_k$ and $a_k+t_k=b_k-t_k+1$, the cardinality of 
$\{(i,j)\in {\mathcal L'}\;|\; i<j\}$ coincides with the cardinality
of $\{(i,j)\in {\mathcal L''}\;|\; i<j\}$, where $\mathcal L''$ has
upper corners $$(a_1+t_1-1,b_1-t_1+1),\ldots,
(a_{k-1}+t_{k-1}-1,b_{k-1}-t_{k-1}+1),(a_k+t_k-1,b_k-t_k),$$
$$(a_k+t_k,b_k-t_k+1),(a_{k+1}+t_{k+1}-1,b_{k+1}-t_{k+1}+1),\ldots,
(a_s+t_s-1,b_s-t_s+1).$$
Equivalently, ${\mathcal L''}$ is
obtained from ${\mathcal L}$ by removing 
$(a_k+t_k-1,b_k-t_k+1)$ and its symmetric point
$(b_k-t_k+1,a_k+t_k-1)$. Hence the codimension of $W$ is $c-1$.

Similarly, if we are in the situation that both (2) and (3) are
verified, we apply Proposition~\ref{codim} to the ladder $\Z$ of case
(1). The codimension of $W$ equals the cardinality
$\{(i,j)\in {\mathcal L'}\;|\; i<j\}$, where
${\mathcal L}'$ is the same as in case (1). 
Since $a_k+t_k=a_{k+1}+t_{k+1}-1$ and $b_k-t_k=b_{k-1}-t_{k-1}+1$, the cardinality of 
$\{(i,j)\in {\mathcal L'}\;|\; i<j\}$ coincides with the cardinality
of $\{(i,j)\in {\mathcal L''}\;|\; i<j\}$, where $\mathcal L''$ is the
same as in case (1). In fact, the angles
$(a_k+t_k,b_k-t_k+1)$ and $(a_{k+1}+t_{k+1}-1,b_{k+1}-t_{k+1}+1)$ are on the same
row. Moreover the angles
$(a_{k-1}+t_{k-1}-1,b_{k-1}-t_{k-1}+1)$ 
and $(a_k+t_k-1,b_k-t_k)$ are on the same
column. We conclude that the codimension of $W$ is $c-1$.

If we are in situation (2), then $a_{k+1}-(a_k+1)=t_k-t_{k+1}$. 
Assume that $(b_k-1)-b_{k-1}>t_k-t_{k-1}$.
Apply Proposition~\ref{codim} to the ladder 
$\Z=\X_1\cup\ldots\cup
\X_{k-1}\cup\X_k^{(1)}\cup\X_{k+1}\cup\ldots\cup\X_s$.
The codimension of $W$ equals the cardinality of 
$\{(i,j)\in {\mathcal L'}\;|\; i<j\}$, where
${\mathcal L}'$ is the ladder with upper corners
$$(a_1+t_1-1,b_1-t_1+1),\ldots,(a_{k-1}+t_{k-1}-1,b_{k-1}-t_{k-1}+1),
(a_k+t_k-1,b_k-t_k)$$
$$(a_{k+1}+t_{k+1}-1,b_{k+1}-t_{k+1}+1),\ldots,
(a_s+t_s-1,b_s-t_s+1).$$
Since $a_k+t_k=a_{k+1}+t_{k+1}-1$, the cardinality of 
$\{(i,j)\in {\mathcal L'}\;|\; i<j\}$ coincides with the cardinality
of $\{(i,j)\in {\mathcal L''}\;|\; i<j\}$, where $\mathcal L''$ is the
same as in case (1). In fact, the angles
$(a_k+t_k,b_k-t_k+1)$ and $(a_{k+1}+t_{k+1}-1,b_{k+1}-t_{k+1}+1)$ are on the same
row. We conclude that the codimension of $W$ is $c-1$.

An analogous argument applies to situation (3), where we consider 
$\Z=\X_1\cup\ldots\cup\X_{k-1}\cup\X_k^{(2)}\cup\X_{k+1}\cup\ldots\cup\X_s$
and observe that $\X_k^{(1)}$ can be disregarded in the codimension
count, as the angles $(a_{k-1}+t_{k-1}-1,b_{k-1}-t_{k-1}+1)$ and
$(a_k+t_k-1,b_k-t_k)$ are on the same column.
\end{proof}

The next theorem is the main result of this paper. 
The idea of the proof is as follows: starting from a ladder pfaffian 
variety $V$, we construct two ladder pfaffian varieties $V'$ and $W$ 
such that $V$ and $V'$ are generalized divisors on $W$. 
Then we show how $V'$ can be obtained from $V$ by an elementary
G-biliaison on $W$. 

\begin{thm}\label{biliaison}
Any ladder pfaffian variety can be obtained from a linear variety
by a finite sequence of ascending elementary G-biliaisons.
\end{thm}

\begin{proof}
Let $V$ be a ladder pfaffian variety. Let $\mathcal Y$ be the
ladder corresponding to $V$,
$$I_V=I_{2t}(Y)=I_{2t_1}(X_1)+\cdots+I_{2t_s}(X_s)\subseteq K[Y].$$
We perform all the linkage steps in $\PP^r=\Proj(K[Y])$.
If $t_1=\ldots=t_s=1$ then $V$ is a linear variety. Therefore we
consider the case when 
$t_k=\mmax\{t_1,\ldots,t_s\}\ge 2$. It follows that
$a_{k+1}-a_k>t_k-t_{k+1}\geq 0$ and $b_k-b_{k-1}>t_k-t_{k-1}\geq 0$,
therefore $a_{k-1}<a_k+1\leq a_{k+1}$ and $b_{k-1}\leq b_k-1<b_{k+1}$. 
Let $\mathcal Y'$ be the ladder with upper corners 
$$(a_1,b_1),\ldots,(a_{k-1},b_{k-1}),(a_k+1,b_k-1), 
(a_{k+1},b_{k+1}),\ldots,(a_s,b_s)$$ 
and let $t'=(t_1,\ldots,t_{k-1},t_k-1,t_{k+1},\ldots,t_s)$.
Let $$\Y'=\X_1\cup\ldots\cup \X_{k-1}\cup \X'_k\cup
\X_{k+1}\cup\ldots \cup \X_s,$$ 
where $\X'_k$ is the square subladder with upper outside  corner 
$(a_k+1,b_k-1)$ (see Figure 2). Notice that all 
the inequalities of Remark~\ref{conv} (2),(3) are satisfied by $\Y',t'$.
Let $V'$ be the ladder pfaffian variety 
with saturated ideal $I_{V'}=I_{2t'}(Y').$
By Proposition~\ref{codim} and Remark~\ref{conv} (1), $V$ and $V'$ have the same
codimension $c$ in $\PP^r=\Proj (K[Y])$. In fact, in both cases $c$
equals the cardinality  of the subset 
$\{(i,j)\in{\mathcal L}\;|\; i<j\}$ where ${\mathcal L}$ is 
the ladder with  upper
corners $(a_1+t_1-1,b_1-t_1+1),\ldots,(a_s+t_s-1,b_s-t_s+1).$

Let $\mathcal Z$ be the ladder obtained from $\Y$ by removing 
$(a_k,b_k)$ and $(b_k,a_k)$ (see Figure 3), and let $u=(t_1,\dots,t_{k-1},t_k,t_k,t_{k+1},\dots,t_s)$. 

\begin{figure}
\centering%
\includegraphics[height=20cm , width=10cm,keepaspectratio]  {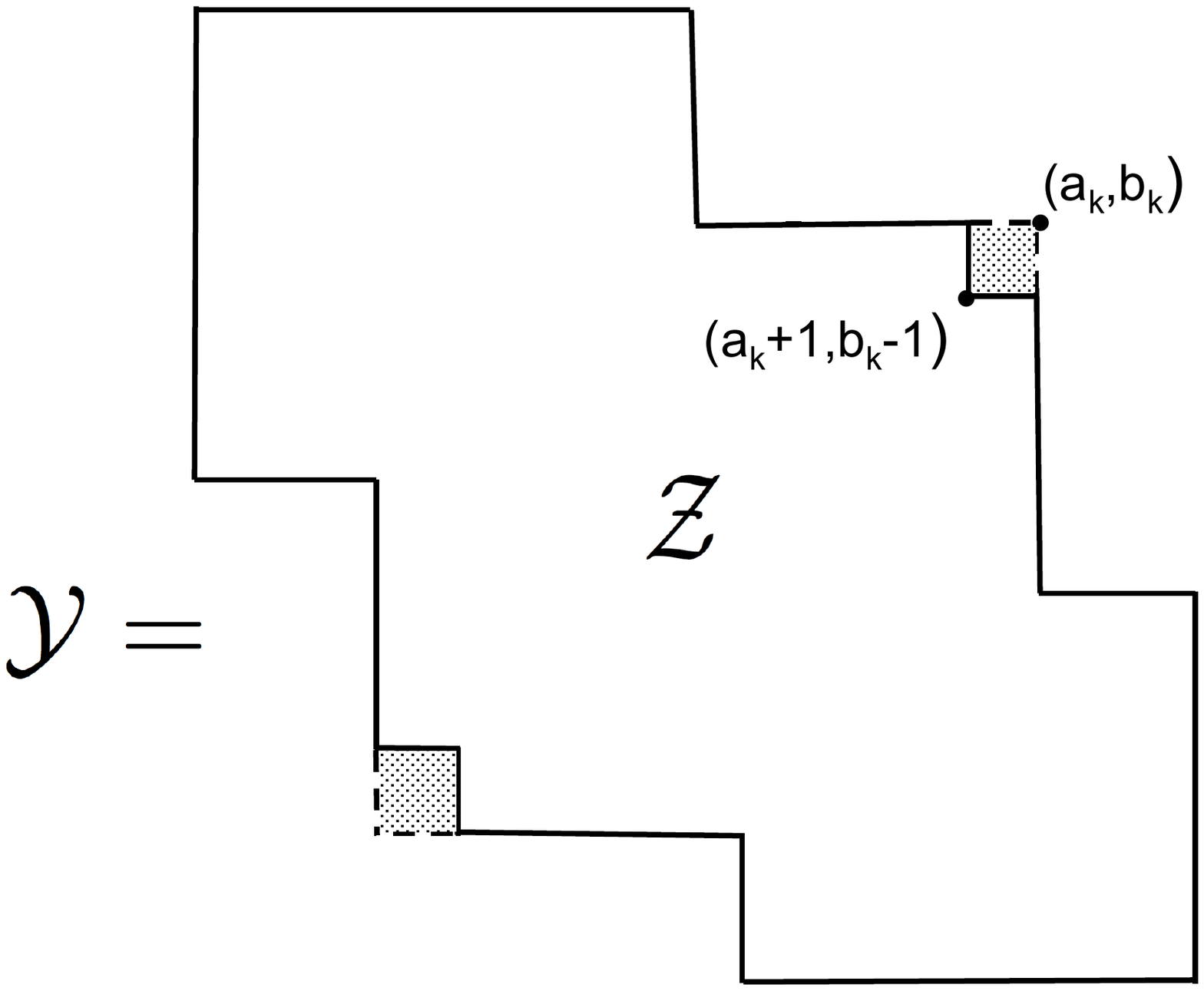}
\caption{}			
\end{figure}

Let $W$ be the ladder pfaffian variety with $I_W=I_{2u}(Z)$. $W$ has codimension $c-1$ in
$\PP^r=\Proj (K[Y])$, by Lemma~\ref{codimW}. 
Clearly $I_V\supseteq I_W$, therefore $V\subseteq W$. We claim
that also $V'\subseteq W$. In order to show that
$I_{2t'}(Y')\supseteq I_{2u}(Z)$, it suffices to show that
$I_{2t_k}(X_k^{(1)})+I_{2t_k}(X_k^{(2)})\subseteq
I_{2(t_k-1)}(X'_k)$. Let $f=[u_1,\dots,u_{2t_k}]$ be a
$2t_k$-pfaffian in $I_{2t_k}(X_k^{(1)})+I_{2t_k}(X_k^{(2)})$, 
with $a_k\le u_1<u_2<\cdots\le u_{2t_k}\le b_k$. If 
$a_k,b_k\not\in\{u_1,\ldots,u_{2t_k}\}$, then 
$f\in I_{2t_k}(X'_k)\subseteq I_{2(t_k-1)}(X'_k)$. 
If $a_k=u_1$ then $b_k\not\in \{u_2,\ldots,u_{2t_k}\}$, since
$f$ does not involve the indeterminate $x_{a_k,b_k}$. 
By expanding $f$ along the $u_1$-th row and column one has
$$f=\sum_{h=1}^{2t_k}\pm[u_1,u_h][u_2,\ldots,\check u_h,\ldots,u_{2t_k}].$$
The notation $[...,\check{u_i},...]$ means that the index $u_i$ is not
involved in the pfaffian.  
Since $[u_2,\ldots,\check u_h,\ldots,u_{2t_k}]\in
I_{2(t_k-1)}(X'_k)$, one has $f\in I_{2(t_k-1)}(X'_k)$.
Similarly, if $u_{2t_k}=b_k$ then
$a_k\not\in\{u_1,\ldots,u_{2t_k-1}\}$, 
and expanding $f$ along the $u_{2t_k}$-th row and column the
conclusion follows.

The variety $W$ is reduced, irreducible, and arithmetically
Cohen-Macaulay by Proposition~\ref{primeCM}. In particular it is
generically Gorenstein. Therefore we can regard $V$ and $V'$ as
generalized divisors on $W$ (see \cite{ha04u2} about the theory
of generalized divisors). Then $V$ and $V'$ are G-bilinked 
on $W$ if and only if $V\sim V'+mH$ for some $m\in\ZZ$, where $H$ 
is the hyperplane section divisor on $W$. This is
in turn equivalent to 
\begin{equation}\label{isom}
\I_{V|W}\cong\I_{V'|W}(-m)\end{equation} 
as subsheaves of the sheaf of total quotient rings of $W$.
In order to construct an isomorphism as~(\ref{isom}), we prove
that 
$$\frac{[a_k,u_2,\ldots,u_{2t_k-1},b_k]}{[u_2,\ldots,u_{2t_k-1}]}=
\frac{[a_k,v_2,\ldots,v_{2t_k-1},b_k]}{[v_2,\ldots,v_{2t_k-1}]}$$
modulo $I_W$, for any choice of $u_i, v_i$ such that  
$a_k<u_2<\ldots<u_{2t_k-1}<b_k$ and $a_k<v_2<\ldots<v_{2t_k-1}<b_k.$
Then multiplication by 
$\frac{[a_k,u_2,\ldots,u_{2t_k-1},b_k]}{[u_2,\ldots,u_{2t_k-1}]}$ for
a fixed choice of $a_k<u_2<\ldots<u_{2t_k-1}<b_k$
yields an isomorphism as~(\ref{isom}).
By Lemma~\ref{relazioni} one has
$$[v_2,\ldots,v_{2t_k-1}][a_k,u_2,\ldots,u_{2t_k-1},b_k]=$$
$$\sum_{h=2}^{2t_k-1} [v_2,\ldots,v_{h-1},
a_k,v_{h+1},,\ldots,v_{2t_k-1}][v_h,u_2,\ldots,u_{2t_k-1},b_k]+$$
$$+\sum_{l=2}^{2t_k-1}(-1)^{l-1}[u_l,a_k,v_2,\ldots,v_{2t_k-1}]
[u_2,\ldots,\check u_l,\ldots, u_{2t_k-1},b_k]+$$
$$(-1)^{2t_k-1}[b_k,a_k,v_2,\ldots,v_{2t_k-1}][u_2,\ldots,u_{2t_k-1}].$$
Since the pfaffians  
$[v_h,u_2,\ldots,u_{2t_k-1},b_k],[u_l,a_k,v_2,\ldots,v_{2t_k-1}]$ are in
$I_W$ for every $h$ 
and $l$, one has that 
$[v_2,\ldots,v_{2t_k-1}][a_k,u_2,\ldots,u_{2t_k-1},b_k]=
-[b_k,a_k,v_2,\ldots,v_{2t_k-1}][u_2,\ldots,u_{2t_k-1}]=
[a_k,v_2,\ldots,v_{2t_k-1},b_k][u_2,\ldots,u_{2t_k-1}]$
modulo $I_W$.
Therefore the isomorphism~(\ref{isom}) holds, and $V$ and $V'$
are G-bilinked on $W$. Repeating this procedure, one eventually
gets to the pfaffians of order 2 of the ladder ${\mathcal L}$ 
with upper corners $\{(a_k+t_k-1,b_k-t_k+1)\;|\;
k=1,\ldots,s\}.$ 
Clearly $I_2(L)=(x_{ij}\;|\; (i,j)\in {\mathcal L}, i<j)$ defines a linear variety.
\end{proof}

Kleppe, Migliore, Mir\'o-Roig, Nagel, and Peterson proved
in~\cite{kl01} that a G-biliaison on an arithmetically
Cohen-Macaulay, $G_1$ scheme can be realized by two Gorenstein
links. In~\cite{ha04u2} Hartshorne generalized their result to a
G-biliaison on an arithmetically Cohen-Macaulay, generically
Gorenstein scheme. Therefore we get the following corollary.

\begin{cor}\label{glicci}
Every ladder pfaffian variety $V$ can be G-linked in
$2(t_1+\ldots+t_s-s)$ 
steps to a linear variety of the same codimension.
In particular, ladder pfaffian varieties are glicci.
\end{cor}

\begin{rmk}
The varieties cut out by pfaffians of fixed size of a
skew-symmetric matrix (i.e. those for which $Y=X$ and
$t_1=\ldots=t_s$) are known to be arithmetically Gorenstein 
(see~\cite{a79} and~\cite{KL}). 
In~\cite{dn98} ideals generated by pfaffians of fixed size in a
ladder are considered, and a characterization is given of the
ones defining arithmetically Gorenstein varieties. The results in \cite{co96}
play a central role in the argument.
It turns out that the arithmetically Gorenstein varieties are
essentially 
only those cut out by
pfaffians of fixed size of a skew-symmetric matrix, and a few
more cases that are directly connected to those. 
Notice that combining Proposition~\ref{local} and the results
in~\cite{dn98}, one easily obtains a characterization of the
arithmetically Gorenstein ladder pfaffian varieties for pfaffians
of mixed size, in terms of the upper outside corners of the
ladder and of the vector $t$. The technique of Theorem~6.3.1 
in ~\cite{go00} applies to this situation.

Arithmetically Gorenstein schemes are known to be glicci (see
Theorem~7.1 of~\cite{ca04u}). Notice however that only very special 
ladder pfaffian varieties are arithmetically Gorenstein.
Moreover, the question of whether every glicci scheme belongs to the
G-biliaison class of a complete intersection remains open.
\end{rmk}

\end{document}